\documentclass[11pt]{amsart}
\usepackage{a4wide}
\usepackage{hyperref,amsmath,amsfonts,mathscinet,amssymb,amsthm,natbib}
\usepackage{paralist} 
\usepackage{xcolor}

\theoremstyle{plain}
\newtheorem{theorem}{Theorem}[section]
\newtheorem{lemma}[theorem]{Lemma}

\newtheorem{proposition}[theorem]{Proposition}

\theoremstyle{definition}
\newtheorem{definition}[theorem]{Definition}

\newtheorem{example}[theorem]{Example}

\theoremstyle{remark}
\newtheorem{remark}[theorem]{Remark}

\usepackage[
	textsize=tiny,
	textwidth=3.2cm,
	colorinlistoftodos,
	backgroundcolor=teal!30!white,
	linecolor=teal!50!white
	]{todonotes}

\numberwithin{equation}{section}

\def\N{\mathbb N}

\def \R{\mathbb{R}}

\begin{document}

\title{Maximal non-compactness of embeddings between sequence spaces}

\author{Anna Kneselová}

\email[A.~Kneselová]{anna.kneselova391@student.cuni.cz}


\begin{abstract}
We study the ball measure of non-compactness 
of various particular instances of embedding operators in sequence spaces. Our first main goal is to find sufficient conditions for the identity operator to be maximally non-compact, that is, that its ball measure of non-compactness coincides with its operator norm. Then, we establish, conversely, a sufficient condition for avoiding maximal non-compactness of such an embedding. Next, we characterize inclusions between two-parameter Lorentz sequence spaces $\ell^{p,q}$ in dependence on the values of $p$ and $q$.
In certain cases, we determine exact values of the norms of the identity operators between these spaces.
Lastly, we apply our general theorems to determining whether the identity operator is maximally non-compact between Lorentz spaces. 
\end{abstract}

\date{\today}

\maketitle

\bibliographystyle{abbrv}

\section*{Introduction}
\addcontentsline{toc}{chapter}{Introduction}

Investigation of compactness of mappings is one of the most important topics in analysis. Its results have proved useful in several fields of mathematics including contemporary functional analysis. It is interesting from the theoretical point of view but, at the same time, it has a wide array of applications. However, especially in infinite-dimensional spaces, the mapping of interest is often not compact.

The difference between compactness and non-compactness of an operator is often too rough, and, therefore, it is worthwhile to study its various refinements. One particular example of such a refinement is offered by studying the so-called \textit{measure of non-compactness}.

The measure of non-compactness, $\alpha(T)$, of a mapping $T$ acting from one quasi-normed space to another, can attain any value between zero and the operator norm $||T||$. The mapping $T$ is compact if and only if $\alpha(T)=0$, at least on complete spaces. At the other end of the range of possible values of $\alpha(T)$, that is, in the case when $\alpha(T)$ coincides with the operator norm, the mapping is called \textit{maximally non-compact}, marking the worst possible case.

The concept of measure of non-compactness is thus a good device for quantifying \textit{how bad} the non-compactness of a mapping is. The idea can be traced back to the 1930 paper by Kuratowski~\cite{Kur:30}. Since its introduction, it has proved very useful and has found many diverse applications, see, for instance, \cite{BG:80,Dar:55,Sad:68}. The concept of measure of non-compactness is intimately connected with the limiting behavior of entropy numbers and $s$-numbers in general~\cite{Car:81, CarTri:80}. 
These advances eventually led to the development of a substantial theory which found its way to important classical monographs such as \cite{EdmEv:18,EdmTri:96,Pie:87,Tri:78}.

An important particular example of an operator whose compactness or non-compactness is extensively studied is the \textit{identity}. In such case, we talk about \textit{embeddings} of the corresponding spaces or classes. The measure of non-compactness of embeddings has been investigated for example in connection with Sobolev embeddings and their applications to properties of solutions to elliptic partial differential equations. The fact that there is an immense interest in these problems can be illustrated by the vast number of papers including some very recent ones. To name just a few, see e.g.~\cite{Bou:19,Hen:03,Lan:22,Lan:24, Lang21}. 

In sharp contrast to the vast existing literature on compactness-related properties of \textit{function} spaces stands the fact that almost nothing is known about analogous problems concerning  \textit{sequence} spaces. This paper is intended as a contribution towards filling this gap. 

Our aim is to study questions about maximal non-compactness of embeddings of various sequence spaces. After stating several general statements, we shall focus on the particular scale of Lorentz sequence spaces.

The paper is structured as follows. In the next section we collect all relevant definitions and fix notation. In Section~\ref{S:results} we state and prove our first set of main results. We focus here on embeddings between general sequence spaces which are endowed with at least a quasi-norm. More precisely, we offer here a sufficient condition for the maximal non-compactness of an embedding between two sequence spaces contained in $c_0$, an analogous result concerning rearrangement-invariant lattices, and also a sufficient condition for an embedding to avoid maximal non-compactness. In our approach to the latter problem we employ the recent concept of \textit{span} of a sequence space.

We then turn our attention to \textit{Lorentz sequence spaces}, which are collections of sequences determined by the two-parameter scale of Lorentz functionals. It is known that the Lorentz functional is always a quasi-norm, which in some special cases turns to a norm or is equivalent to a norm. Lorentz spaces constitute a very interesting collection of sequence spaces, even slightly richer than their companions containing functions (note that, for example, spaces $\ell^{\infty,q}$ with finite $q$ are nontrivial). Let us recall that Lorentz function spaces have been treated in various monographs, see e.g.~\cite{BS87,Pic:13,Ste:71}.

% In the first part we introduce the definition of measure of non-compactness $\alpha(T)$, which we get from the article \cite{Lang21}. 

% We also state and prove some of its basic properties, which follow from those of metric spaces and the operator norm, and are mentioned in the article without the proof.

% In the second part we focus on generalisation of the example in the introduction of the article \cite{Lang21}. We formulate the prerequisites and state some general theorems which are addressing the question whether an identity operator between two embedded sequence spaces is maximally non-compact.

% In the final third part we study the Lorentz sequence spaces 
% $\ell^{p,q}$ and some of their basic properties.

In order to enable application of our general results to Lorentz spaces, we begin by showing that, for an arbitrary combination of $p,q \in(0,\infty]$ such that $\min\{p,q\} < \infty$, the space $\ell^{p,q}$ is a subset of $c_{0}$ by some elementary means. Next, we concentrate on characterizing the inclusions $\ell^{p_{1},q_{1}} \hookrightarrow \ell^{p_{2},q_{2}}$ between the Lorentz sequence spaces depending on the values of $p_1, p_{2}$ and $q_{1}, q_{2}$.
We extend the inequalities and inclusions mentioned in \cite{Kato76}, and we fill in the details to complete the proofs and give a comprehensive picture.
In most cases, we furthermore determine the exact values of the norms of the identity operators between these spaces. For the remaining cases, specifically, for certain parameter ranges where $q_2 < \infty$, the exact value of the operator norm remains unknown, and we formulate this as an open question later in the text.
This knowledge is very useful in our last set of results, in which we determine whether embedding relations between Lorentz sequence spaces are maximally non-compact.

For some other related results, see e.g.~\cite{Ban:92,Mal:06,Lop:02,Tro:04,Yer:94}.

\section{Preliminaries}
Throughout this paper we will consider that all (quasi)-normed linear spaces are defined over $\R$.

\begin{definition}[Ball measure of non-compactness]
Let $X$, $Y$ be (quasi)-normed linear spaces and let $T\colon X \rightarrow Y$ be a bounded mapping defined on $X$ and taking values in $Y$. The \emph{ball measure of non-compactness} $\alpha(T)$ of $T$ is defined as the infimum of radii $r > 0$ for which there exists a finite set of balls in $Y$ of radii $r$ that covers $T(B_{X})$, where $B_{X}$ is the closed unit ball in $X$. In other words,
$$\alpha(T) = \inf\{r>0: \exists \ m \in \N, \ \exists \ y_1, \dots, y_m \in Y \ \text{such that } T(B_{X}) \subset \bigcup_{i = 1}^m (y_{i} + rB_{Y}) \}.$$
\end{definition}

Let us recall that, given (quasi)-normed linear spaces  $X$, $Y$ and a bounded operator $T\colon X \rightarrow Y$ defined on $X$ and taking values in $Y$, we define the operator norm $\|T\|$ of $T$ as
$$\|T\| = \sup_{x\in B_{X}} \|T(x)\|.$$

We will formulate an elementary but important inequality estimating the ball measure of non-compactness.

    \begin{proposition}Let $X$ and $Y$ be (quasi)-normed linear spaces and let $T\colon X \rightarrow Y$ be a bounded linear operator defined on $X$ and taking values in $Y$. Then
$$0 \leq \alpha(T) \leq \|T\|.$$
\end{proposition}
The proof directly follows from the definition of the ball measure of non-compactness.

%\begin{proof} We will prove each of the inequalities separately: We can see that $\alpha(T) \geq 0$ easily from the definition of the ball measure of non-compactness, because infimum of positive radii $r$ is always greater or equal to zero.Take $\varrho > \|T\|$ arbitrarily. Then for every $x\in B_{X}$ one has $\|Tx\|\leq \|T\|<\varrho$, which in turn implies that $Tx \in B_{Y}(0,\varrho)$. Since $x$ was arbitrary, this implies that $$T(B_{X}) \subset \bigcup_{i = 1}^{m=1} (0 +\varrho B_{Y}).$$The last inclusion is true for all $\varrho > \|T\|$. Consequently, $\alpha(T) \leq \|T\|$.
%\end{proof}

\begin{definition}[Maximal non-compactness]
Let $X$ and $Y$ be (quasi)-normed linear spaces and let $T\colon X \rightarrow Y$ be a bounded operator defined on $X$ and taking values in $Y$. We say that $T$ is \emph{maximally non-compact} if
$$\alpha(T) = \| T \|.$$
    \end{definition}

The following statement gives us the relationship between the value of the ball measure of non-compactness of an operator and compactness itself.

\begin{proposition}\label{Prop-BP&MN}Let $X$ and $Y$ be (quasi)-normed linear spaces and let $T\colon X \rightarrow Y$ be a bounded linear operator defined on $X$ and taking values in $Y$.
\begin{enumerate}[1.]
    %\item $T$ is compact if and only if for every bounded $A\subset X$ the set $T(A)$ is relatively compact in $Y$.
    \item If $T$ is compact, then $\alpha(T) = 0$.
    \item If $Y$ is a complete space and $\alpha(T) = 0$, then $T$ is compact.
\end{enumerate}
\end{proposition}
Variants of Proposition~\ref{Prop-BP&MN} are well known. The assertion can be easily obtained using standard results in analysis, in particular, the definition of a compact linear operator and the equivalence of total boundedness and relative compactness in complete metric spaces.

\begin{definition}[Embedding]\label{D:embedding}
    Assume that $(\ell, \|\cdot\|_{\ell})$  and $(w, \|\cdot\|_{w})$ are sequence (quasi)-normed linear spaces. We say that $\ell$ is \emph{embedded} into $w$ (we shall denote this as $\ell \hookrightarrow w$) if $\ell$ is a linear subspace of $w$ and the identity operator $I$ is bounded in the sense that there exists a constant $C \geq 0$ such that for every $a\in \ell$ one has $\|a\|_{w} \leq C \|a\|_{\ell}.$
\end{definition}

We denote by $\mathcal P(\mathbb N)$ the power set of the set $\mathbb N$ of all natural numbers. We endow $\N$ with the usual $\sigma$-finite counting measure $m\colon\mathcal P(\mathbb N)\to[0,\infty]$.

Below, we define a decreasing rearrangement for sequences. We use the definition from \cite[Chapter 2, Definition~1.1 and Definition~1.5; page 36--39]{BS87} and reformulate it for the special case where the underlying measure space is
$(\mathbb{N}, \mathcal P(\mathbb N), m)$.

\begin{definition}[Distribution function and decreasing rearrangement]\label{Def: a^*}
For a sequence $a = \{a_n\}_{n = 1}^{\infty} \subset \mathbb{R}$, we define its \emph{distribution function} $m_{a}\colon[0, \infty) \rightarrow [0, \infty]$ by
$$m_{a}(\omega) = m\{n \in \mathbb{N}\colon |a_{n}| > \omega\}.$$
Then, for $t \geq 0$, we define the \emph{decreasing rearrangement} as
$$a^{*}(t) = \inf\{ \omega > 0 \colon m_{a}(\omega) \leq t \}$$
and the \emph{decreasing rearrangement $a^{*} =\{a^{*}_n\}_{n = 1}^{\infty}$} of the sequence $a$ as
$$a_{n}^{*}= a^{*}(t) = \inf\{ \omega > 0 \colon m_{a}(\omega) \leq n-1 \}$$ 
for $n-1 \leq t < n$.
\end{definition}

When equipped with the decreasing rearrangement, we are in a position to introduce the crucial notion of a rearrangement-invariant lattice, which will have decisive significance for determining the maximal non-compactness of an embedding into the space $c_0$.

\begin{definition}[Rearrangement-invariant lattice]\label{Def: r.i.lattice}
Let $(\ell, \|\cdot\|_{\ell})$ be a (quasi)-normed sequence space. We say that $\ell$ is a~\emph{rearrangement-invariant lattice} if it satisfies the following axioms:
\begin{enumerate}[1.]
    \item If $b \in \ell$ and $a$ is a sequence such that $|a_{n}| \leq |b_{n}|$ for all $n \in \mathbb{N}$, then $a \in \ell$ and $\|a\|_{\ell} \leq \|b\|_{\ell}$.
    \item If $b \in \ell$ and $a$ is a sequence such that $a^* = b^*$, then $a \in \ell$ and $\|a\|_{\ell} = \|b\|_{\ell}$.
\end{enumerate}
\end{definition}

We now introduce a notion which, as we will show later, will be important in determining whether an embedding operator is not maximally non-compact.

\begin{definition}[Sequence space span]
Let $\ell$ be a (quasi)-normed sequence space. We define the \emph{span} $\sigma_{\ell}$ of $\ell$ as
\begin{equation}\label{span def}
    \sigma_{\ell} = \sup_{y\in B_{\ell}}\left(\sup_{n\in \mathbb{N}}y_{n}- \inf_{n\in \mathbb{N}}y_{n} \right).
\end{equation}
\end{definition}

We will recall the definitions of some familiar sequence spaces, namely, $\ell^{p}$ for $1 \leq p \leq \infty$ and $c_{0}$. One has
\begin{align*}
&\ell^{p} = \{ \{x_{n}\}_{n=1}^{\infty} \subset \mathbb{R} ;
        \sum_{n=1}^{\infty} |x_{n}|^p  < + \infty\}\text{ with norm } \|x \|_{p} = \left(\sum_{n=1}^{\infty} |x_{n}|^{p}\right)^{\frac{1}{p}}\text{ and } p< \infty,
    \\
&\ell^{\infty} = \{ \{x_{n}\}_{n=1}^{\infty} \subset \mathbb{R}; \sup_{n \in \mathbb{N}}|x_{n}| < + \infty\}\text{ with norm } \|x\|_{\infty}= \sup_{n \in \mathbb{N}}|x_{n}|\text{ and } p = \infty,
    \\
&c_{0} = \{ \{x_{n}\}_{n=1}^{\infty} \subset \mathbb{R} ; \;
        \lim_{n\to \infty} x_{n} = 0 \}\text{ with norm }\|x\|_{\infty}.
\end{align*}

We shall now introduce the Lorentz sequence spaces, a pivotal example of a scale of sequence spaces, whose rich and diverse structure will enable us to illustrate our general results.

\begin{definition}[Lorentz sequence spaces $\ell^{p,q}$] Let  $p,q \in (0, \infty]$. We define \emph{Lorentz sequence space} $\ell^{p,q}$ as a space of all sequences $a = \{a_{n}\}_{n=1}^{\infty} \subset \mathbb{R}$, for which the value  $\|a\|_{p,q}$, defined as
\begin{equation*}
        \|a\|_{p,q}=
        \begin{cases}
        \left(\displaystyle\sum_{n=1}^{\infty} (a_{n}^{*})^{q} n^{\frac{q}{p} - 1}\right)^{\frac{1}{q}}, & \text{if $0 < q < \infty $ and $0 < p \leq \infty$},\\\\
        \displaystyle\sup_{n\in \mathbb{N}} \{n^{\frac{1}{p}}a_{n}^{*}\}, & \text{if $q = \infty$ and $0 < p \leq \infty$},
        \end{cases}
\end{equation*}
is finite.
\end{definition}

\begin{remark} Throughout, we adopt the convention $\frac{1}{\infty} = 0$. 
\end{remark}

\begin{remark} We recall some of the basic properties of Lorentz sequence spaces which follow directly from the definition for all $p,q \in (0, \infty]$, namely
    \begin{enumerate}
        \item $\|a\|_{p,q} =\| n^{\frac{1}{p} - \frac{1}{q}}a_{n}^{*} \|_{\ell^{q}}$,
        \item $\ell^{p,p} = \ell^{p}$.
        
    \end{enumerate}
\end{remark}

\begin{remark} 
    Standard techniques show that $\|\cdot\|_{p,q}$ is a quasi-norm whenever $p,q\in(0,\infty]$. Moreover, it can be shown that, for $1 \leq q \leq p\leq \infty$, it is a norm. For further details, see e.g.~\cite[p.~76, Proposition 1]{Kato76}.
\end{remark}

% If $1 \leq q \leq p\leq \infty$, then $\|\cdot\|_{p,q}$ is a norm. If $0 < p<q\leq \infty$, $\|\cdot\|_{p,q}$ is not a norm but it is a quasi-norm satisfying, for any $a,b \in \ell^{p,q}$,
% $$\|a+b\|_{p,q} \leq 2^{1/p}(\|a\|_{p,q}+\|b\|_{p,q}).$$
% The proof can be found in \cite[p.~76, Proposition 1]{Kato76}.

\begin{remark}
Any Lorentz sequence space $\ell^{p,q}$, $p,q \in (0
, \infty]$, trivially satisfies the axioms in Definition~\ref{Def: r.i.lattice}. Therefore, it is a rearrangement-invariant lattice.
\end{remark}

We will now formulate two well-known elementary observations which will come in handy in the sequel. 

\begin{lemma}\label{a+b q>1} Let $a,b \geq 0$ and $q\geq1$, then 
    \begin{equation*}
          (a+b)^{q}\leq 2^{q-1}(a^{q}+b^{q}).
    \end{equation*}
\end{lemma}

%\begin{proof}We know that the function $f(t) = t^{q}$ is convex for $q\geq1$, so $f(\lambda t + (1-\lambda)t) \leq \lambda f(t)+ (1-\lambda)f(t)$. From this it already follows that\begin{equation*}\left(\frac{1}{2}\right)^{q}(a+b)^{q} =\left(\frac{1}{2}a+\frac{1}{2}b\right)^{q}\leq \frac{1}{2}(a^{q}+b^{q})\end{equation*}and we are done.\end{proof}

\begin{lemma}\label{a+b q<1} Let $a,b \geq 0$ and $0<q<1$, then 
    \begin{equation*}
          (a+b)^{q}\leq a^{q}+b^{q}.
    \end{equation*}
\end{lemma} 
%\begin{proof}It is equivalent to show\begin{equation*}\left(\frac{a}{b}\right)^{q}+1 \geq\left(\frac{a}{b}+1\right)^{q}\end{equation*}We will prove this by showing that the function\begin{equation*}f(x) = x^{q} +1 -\left(x+1\right)^{q}\end{equation*}is nonnegative and increasing for $x\geq0$. We have $f(0) = 0$ and\begin{equation*}f^{\prime}(x) = qx^{q-1} - q\left(x+1\right)^{q-1}.\end{equation*}Then $f^{\prime}(x) \geq 0$ for all $x\geq0$, because $q-1<0$ and $x^{q-1}\geq\left(x+1\right)^{q-1}$.\end{proof}

\section{Main results}\label{S:results}

We begin with a sufficient condition for an embedding between fairly general sequence spaces to be maximally non-compact. For the key idea of our approach to work, it is important that the spaces are contained in $c_0$.

\begin{proposition}\label{prop-MN-||I||<1}
Let $\ell$ and $w$ be sequence (quasi)-normed linear spaces satisfying
\begin{enumerate}
    \item $\ell,\ w \subset c_{0}$,
    \item for all $x \in w$ and for all $k\in\mathbb{N}$ it holds that $|x_{k}|\leq\|x\|_{w}$,
    \item the embedding $ I\colon \ell \rightarrow w$ satisfies
    $0 <\|I\| \leq 1$, and $e^{j} \in B_{\ell}$ for every $j \in \mathbb{N}$.
\end{enumerate}
Then, $I$ is maximally non-compact.
\end{proposition}

\begin{proof} 
Notice that the assumptions (2) and (3) already imply that $\|I\| = 1$. Indeed, since $e^{j} \in B_{\ell}$ and by using the second assumption for $x = e^{j}$ and $k = j$, we have
$$ 1 = |e^{j}_{j}| \leq \|e^{j}\|_{w} = \|I(e^{j})\|_{w} \leq \|I\| \|e^{j}\|_{\ell} \leq \|I\| \leq 1, $$
which yields $\|I\| = 1$.

We will prove the proposition by contradiction.
Let us assume that $I$ is not maximally non-compact. Then
$\alpha(I) < 1$, and we take $\varrho \in (\alpha(I), 1)$.
Thanks to the definition of the ball measure of non-compactness, we find $m \in \mathbb{N}$ and $y^{1},\dots, y^{m} \in w$ such that
$$I(B_{\ell}) \subset \bigcup_{i = 1}^m (y^{i} + \varrho B_{w}).
$$
Note that
$y^{i} \in w\subset c_{0}$ for each $i \in \{1, \dots , m \}$. So, owing to the definition of a limit, we can find $j_{0} \in \mathbb{N}$ satisfying $|y^{i}_{j}|< 1-\varrho$ for all $j \geq j_{0}$ and all $i \in \{1, \dots, m \}$. Now, for every $ j \geq j_{0}$, the sequence $e^{j} = (0,\dots,\ 0, \ 1,\ 0,\dots)$ satisfies $e^{j} \in I(B_{\ell})$,
but $e^{j} \notin (y^{i} + \varrho B_{w})$ for any
$i \in \{1, \dots, m \}$,
because, for every $ z \in (y^{i} + \varrho B_{w})$, we have
\begin{align*}
    |z_{j}|  & \leq |z_j - y_{j}^{i}| + |y_{j}^{i}| \leq \|z - y^{i}\|_{w} + |y_{j}^{i}| \\
    &<  \varrho + 1 - \varrho = 1.
\end{align*}
The second inequality follows from the second assumption, and the third one from the fact that $ z \in (y^{i} + \varrho B_{w})$ and our previous estimate on $y_{j}^{i}$.
So, we get a contradiction. Hence $\alpha(I) = 1$, and $I$ is maximally non-compact.
\end{proof}

%\begin{proof} We will prove the proposition by contradiction.Let us assume that $I$ is not maximally non-compact. Then$\alpha(I) < \|I\|$, and we take $\varrho \in (\alpha(I), \|I\|)$.Thanks to the definition of the ball measure of non-compactness, we find $m \in \mathbb{N}$ and $y^{1},\dots, y^{m} \in w$ such that $$I(B_{\ell}) \subset \bigcup_{i = 1}^m (y^{i} + \varrho B_{w}).$$Note that$y^{i} \in w\subset c_{0}$ for each $i \in \{1, \dots , m \}$. So, owing to the definition of a limit, we can find $j_{0} \in \mathbb{N}$ satisfying $|y^{i}_{j}|< \|I\|-\varrho$ for all $j \geq j_{0}$ and all $i \in \{1, \dots, m \}$. Now, for every $ j \geq j_{0}$, the sequence $\|I\| e^{j} = (0,\dots,\ 0, \ \|I\|,\ 0,\dots)$ satisfies $\|I\| e^{j} \in I(B_{\ell})$,but $\|I\| e^{j} \notin (y^{i} + \varrho B_{w})$ for any $i \in \{1, \dots, m \}$, because, for every $ z \in (y^{i} + \varrho B_{w})$, we have \begin{align*} |z_{j}|  & \leq |z_j - y_{j}^{i}| + |y_{j}^{i}| \leq \|z - y^{i}\|_{w} + |y_{j}^{i}| \\ &<  \varrho + \|I\| - \varrho = \|I\|. \end{align*} The second inequality follows from the second assumption, and the third one from the fact that $ z \in (y^{i} + \varrho B_{w})$ and our previous estimate on $y_{j}^{i}$.So, we get a contradiction. Hence $\alpha(I) =\|I\|$, and $I$ is maximally non-compact.\end{proof}

We will now present an example which demonstrates that the second assumption of Proposition~\ref{prop-MN-||I||<1} cannot be dropped without violating the validity of the statement.

\begin{example} 
Let $p\in [1,\infty)$ and $\alpha \in (0, 1)$. Set $\ell=\ell^{p}$ and let $w_\alpha$ be the weighted $\ell^{p}$ space endowed with the norm 
$$
    \|x \|_{w_\alpha} = \left( |x_{1}|^{p} + \alpha\sum_{n=2}^{\infty} |x_{n}|^{p} \right)^{\frac{1}{p}}.
$$
Then $\ell$, $w_\alpha$ and embedding $ I\colon \ell \rightarrow w_\alpha$ satisfy all but the second assumption of Proposition~\ref{prop-MN-||I||<1}, however, $I$ is not maximally non-compact, and 
$\alpha(I) \leq \alpha^{\frac{1}{p}} < \|I\|$.

Indeed, this can be verified by the following reasoning.
    \begin{enumerate}[1.]
    \item One has $\ell, w_\alpha \subset c_{0}$ due to the necessary condition for convergence of the sum.
    \item The space $w_\alpha$ does not satisfy the condition that for all $x \in w_\alpha$ and for all $k\in\mathbb{N}$ it holds that $|x_{k}|\leq\|x\|_{w_\alpha}$. For example, take $ x = e^{2} = (0,1,0,0,\dots)$, then 
    $1 = |x_{2}|>\|x\|_{w_\alpha} =\alpha^{\frac{1}{p}}$. 
    \item We will verify the third assumption of Proposition~\ref{prop-MN-||I||<1}.
    \begin{enumerate}[a)]
    \item  It is easy to see that $e^{j} \in B_{\ell}$ for every $j \in \mathbb{N}$.
    \item  For all $x \in \ell^{p}$ we have $\|x\|_{w_\alpha} \leq \|x\|_{p}$ and $\|e^{1}\|_{p} = \|e^{1}\|_{w_\alpha} = 1$, hence $\|I\| = 1$.
    \end{enumerate}
\end{enumerate}

We show that $\alpha(I) \leq \alpha^{\frac{1}{p}}$. Let $\varrho > \alpha^{\frac{1}{p}}$ be an arbitrary radius. 
For any sequence $x \in B_{\ell^{p}}$, let us define its projection onto the first coordinate axis as $Px = x_{1} e^{1}$. 
Since $x \in B_{\ell^{p}}$, we obtain $|x_{1}|^{p} + \sum_{n=2}^{\infty} |x_{n}|^{p} \leq 1$, which directly implies that
\begin{equation*}
    \sum_{n=2}^{\infty} |x_{n}|^{p} \leq 1 - |x_{1}|^{p}.
\end{equation*}
Using this bound, we evaluate the distance between the sequence $x$ and its projection $Px$ in the target space $w_{\alpha}$. We have
$$ \|x - Px\|_{w_\alpha} = \left(\alpha \sum_{n=2}^{\infty} |x_{n}|^{p} \right)^{\frac{1}{p}} \leq \left( \alpha (1 - |x_{1}|^{p}) \right)^{\frac{1}{p}} \leq \alpha^{\frac{1}{p}}. $$
This indicates that every element in $I(B_{\ell^{p}})$ lies within a maximum distance of $\alpha^{\frac{1}{p}}$ from its projection on the first coordinate in the target space $w_{\alpha}$. 
Furthermore, because $x \in B_{\ell^{p}}$ guarantees that $|x_{1}| \leq 1$, the projection $Px = x_{1} e^{1}$ will always belong to the set $E =\{\lambda e^{1} : \lambda \in [-1, 1]\}$.

Notice that $E$ is a closed and bounded subset of the $1$-dimensional subspace $\text{span}(e^{1}) = \{\lambda e^{1} : \lambda \in \R\}$. Therefore, it is a compact set. Hence, it can be covered by a finite number of open balls of any strictly positive radius. Let us define a new radius $r = \varrho - \alpha^{\frac{1}{p}} > 0$. We can cover the set $E$ with a finite number of balls of radius $r$ centered at points $y^{1}, \dots, y^{m} \in \text{span}(e^{1})$. 

Fix $x\in B_{\ell^{p}}$. Because $Px \in E$, and $E$ is completely covered by this finite collection of balls, the projection $Px$ must fall inside at least one of them. Consequently, there exists some index $i \in \{1, \dots, m\}$ such that
$$ \|Px - y^{i}\|_{w_\alpha} < r. $$
Finally, we get
$$ \|x - y^{i}\|_{w_\alpha} \leq \|x - Px\|_{w_\alpha} + \|Px - y^{i}\|_{w_\alpha} < \alpha^{\frac{1}{p}} + \left( \varrho - \alpha^{\frac{1}{p}} \right) = \varrho. $$
This demonstrates that for any $x \in B_{\ell^{p}}$, the element $I(x)$ is contained within a ball of radius $\varrho$ centered at one of the points $y^{1}, \dots, y^{m}$. Thus, the finite union of these $m$ balls covers $I(B_{\ell^{p}})$. Since this finite cover can be constructed for any $\varrho > \alpha^{\frac{1}{p}}$, the definition of the ball measure of non-compactness dictates that $\alpha(I) \leq \alpha^{\frac{1}{p}}$. 
Because we assumed $\alpha \in (0, 1)$, we conclude that $\alpha(I) \leq \alpha^{\frac{1}{p}} < 1 = \|I\|$. Therefore, the embedding operator $I$ is not maximally non-compact, proving that the second assumption of Proposition~\ref{prop-MN-||I||<1} cannot be omitted.
\end{example}

\bigskip

 We shall now state our first main result, in which we formulate a sufficient condition for the maximal non-compactness of an embedding of rearrangement-invariant lattices contained in $c_0$.

\begin{theorem}\label{THM:MN-c_0 subspaces}
Let $\ell$  and $w$ be (quasi)-normed linear sequence spaces satisfying
\begin{enumerate}
    \item $\ell \hookrightarrow w \subset c_{0}$,
    \item  $\ell,\ w $ are rearrangement-invariant lattices.
\end{enumerate}
Then, the embedding operator $I$ is maximally non-compact.
\end{theorem}

\begin{proof} The assertion of the theorem clearly holds for the case when $\ell$ is trivial. Therefore, we can safely assume that $\|I\| > 0$.

We will prove the claim by contradiction.
Let us assume that $I$ is not maximally non-compact. Then
$\alpha(I) < \|I\|$. We take $\varrho \in (\alpha(I), \|I\|)$ and find $\lambda \in (0,1)$ such that
\begin{equation}
    \frac{\varrho}{1-\lambda} < \|I\|.
\end{equation}
From the definition of $\|I\|$, we find $x \in B_{\ell}$ satisfying
\begin{equation}\label{ineq: THM -MN-c_0- x}
    \|x\|_{w} > \frac{\varrho}{1-\lambda}.
\end{equation}
Then, by the definition of the ball measure of non-compactness, we find $m \in \mathbb{N}$ and $y^{1},\dots, y^{m} \in w$ such that
\begin{equation}\label{contradiction-coverage}
    I(B_{\ell}) \subset \bigcup_{i = 1}^m (y^{i} + \varrho B_{w}).
\end{equation}
Now, $y^{i} \in w\subset c_{0}$ for each $i \in \{1, \dots , m \}$.
Since the sequences $y^i$ vanish at infinity, we can inductively construct an increasing sequence of indices $j_1 < j_2 < \dots$ such that for every $k \in \mathbb{N}$ with $x_k^* > 0$, and for all $i \in \{1, \dots, m\}$, we have
\begin{equation}\label{ineq-THM-MN-c_0-y^i<epsilon}
    |(y^{i})_{j_k}| \le \lambda x_{k}^{*}.
\end{equation}
For indices $k$ where $x_k^* = 0$, we impose no restriction and simply choose any $j_k > j_{k-1}$ (setting $j_0 = 0$).

Let us now define a new sequence by
\begin{equation}
    a = \sum_{k = 1}^{\infty}x_{k}^{*}e^{j_{k}}.
\end{equation}
Then $a^*=x^*$. Because $x \in B_\ell$, and both $\ell$ and $w$ satisfy the second axiom of a rearrangement-invariant lattice (Definition~\ref{Def: r.i.lattice}), it immediately follows that $a \in \ell$ and $a \in w$, with $\|a\|_{\ell} = \|x^{*}\|_{\ell} = \|x\|_{\ell}$ and $\|a\|_{w} = \|x^{*}\|_{w} = \|x\|_{w}$. Consequently, $a \in B_{\ell}$.

Now, for every $i \in \{1, \dots, m \}$ and every $k \in \mathbb{N}$, we evaluate $|a_{j_k} - y_{j_k}^i|$.
If $x_k^* > 0$, we use inequality \eqref{ineq-THM-MN-c_0-y^i<epsilon} to obtain
\begin{equation*}
    |a_{j_{k}} - y_{j_{k}}^{i}| = |x_{k}^{*} - y_{j_{k}}^{i}| \ge x_{k}^{*} - |y_{j_{k}}^{i}| \ge x_{k}^{*} - \lambda x_{k}^{*} = (1-\lambda)x_{k}^{*}.
\end{equation*}
If $x_k^* = 0$, we trivially have
\begin{equation*}
    |a_{j_{k}} - y_{j_{k}}^{i}| = |0 - y_{j_{k}}^{i}| = |y_{j_{k}}^{i}| \ge 0 = (1-\lambda)x_{k}^{*}.
\end{equation*}
Thus, the inequality $|a_{j_{k}} - y_{j_{k}}^{i}| \ge (1-\lambda)x_{k}^{*}$ holds for all $k \in \mathbb{N}$.

Since $w$ is a rearrangement-invariant lattice, this pointwise bound allows us to estimate the norm. Utilizing inequality \eqref{ineq: THM -MN-c_0- x}, we get
\begin{flalign*}
    \|a - y^{i}\|_{w}
    & \ge \displaystyle \|\sum_{k=1}^{\infty} e^{j_{k}} (a_{j_k} - y_{j_{k}}^{i})\|_{w} \\
    & \ge \displaystyle \|\sum_{k=1}^{\infty} e^{j_{k}} (1-\lambda)x_{k}^{*}\|_{w} = \|(1-\lambda)x^{*}\|_{w}\\
    & = (1-\lambda) \|x\|_{w} > \varrho.
\end{flalign*}
We showed that $a \notin (y^{i} + \varrho B_{w})$ for all $i \in \{1, \dots, m \}$. This gives a contradiction with \eqref{contradiction-coverage} for any choice of $\varrho \in (\alpha(I), \|I\|)$. Therefore, $I$ is maximally non-compact and $\alpha(I) = \|I\|$.
\end{proof}

Our next main result gives a~general sufficient condition for an embedding operator into $\ell^{\infty}$ to avoid maximal non-compactness.
The following theorem can be viewed as a sequence-space analogue of \cite[Proposition~5.1]{Lan:22}, where the ``span method” is developed for Sobolev embeddings.

\begin{theorem}\label{THM-not max noncompact}   
Let $\ell$ be a (quasi)-normed linear sequence space satisfying $I \colon \ell \to\ell^{\infty}$ and
\begin{equation}\label{eq-sup-inf}
       \|I\|\le\sigma_{\ell}<2\|I\|.
\end{equation}
Then, the embedding operator $I\colon \ell \rightarrow \ell^{\infty}$ is not maximally non-compact. Moreover, $\alpha(I) \leq \frac{\sigma_{\ell}}{2}$.

\end{theorem}

\begin{proof}
    Denote $\varrho \in (\frac{\sigma_{\ell}}{2}, \|I\|)$ and consider $m \in \mathbb{N}$ such that 
    \begin{equation}\label{eq01:THM}
        \left( 1+ \frac{1}{m}\right)\frac{\sigma_{\ell}}{2 }< \varrho.
    \end{equation}
    Define $\lambda_{k} = \frac{\sigma_{\ell} k}{2m}$ for $k = -m, \dots , m$ and let $c^{k}$ be a constant sequence defined by $(c^{k})_{j} = \lambda_{k}$ for every $j\in \mathbb{N}$ and $k = -m, \dots , m$. We will show that
    
    \begin{equation}\label{eq02:THM}
        I(B_{\ell}) \subset \bigcup_{i = -m}^m (c^{i} + \varrho B_{\ell^{\infty}}),
    \end{equation}
    which will imply that $\alpha(I) \leq \varrho < \|I\|$. Assume $y \in B_{\ell}$. Then $y \in B_{\ell^\infty}(0,\|I\|)$ and $|y_{j}| \leq \|I\| \leq \sigma_{\ell}$ for every $j\in \mathbb{N}$. Now, from \eqref{span def} it follows that $ \sup y - \inf y \leq \sigma_{\ell}$. We shall now distinguish three cases.
    
    \begin{enumerate}[a)]
        \item If $\inf y = - \sigma_{\ell}$, then $\sup y \in [- \sigma_{\ell},0]$. Thus,  $y_{j} \in [-\sigma_{\ell}, 0]$ for each $j\in \mathbb{N}$, and we claim that $y \in c^{-m} + \varrho B_{\ell^{\infty}}$. 
        Indeed, since $\varrho>\frac{\sigma_{\ell}}{2}$ and for every $j\in \mathbb{N}$ we have $(c^{-m})_j = \lambda_{-m} = \frac{\sigma_{\ell} (-m)}{2m}=-\frac{\sigma_{\ell}}{2}$, the claim follows.
        
        \item If $\inf y \in (- \sigma_{\ell},0]$, then there is a unique $k \in \{-m+1,\dots , m\}$ such that $\inf y + \frac{\sigma_{\ell}}{2} \in (\lambda_{k-1}, \lambda_{k}] = ( \frac{\sigma_{\ell} (k-1)}{2m}, \frac{\sigma_{\ell} k}{2m}] \subset (- \frac{\sigma_{\ell}}{2}, \frac{\sigma_{\ell}}{2}]$. Then, by the choice of $\varrho > \frac{\sigma_{\ell}}{2}$ and inequality \eqref{eq-sup-inf},
        \begin{equation*}
            \lambda_{k} + \varrho > \lambda_{k} + \frac{\sigma_{\ell}}{2} \geq \inf y + \sigma_{\ell} \geq \sup y.
        \end{equation*}
        Here, we get the second inequality from $$\inf y + \sigma_{\ell} \in (\lambda_{k-1} + \frac{\sigma_{\ell}}{2}, \lambda_{k} + \frac{\sigma_{\ell}}{2}].$$
        On the other hand, using the definition of $\lambda_{k}$ and inequality \eqref{eq01:THM}, we arrive at
        \begin{equation*}
           \inf y > \lambda_{k-1} - \frac{\sigma_{\ell}}{2} =\frac{\sigma_{\ell} k - \sigma_{\ell} -\sigma_{\ell} m}{2m} = \lambda_{k} - \left( 1+ \frac{1}{m}\right)\frac{\sigma_{\ell}}{2 } > \lambda_{k}-\varrho.
        \end{equation*}

        \item If $\inf y \in (0, \sigma_{\ell}]$, then $y_{j} \in [0,\sigma_{\ell}]$ for each $j\in \mathbb{N}$, hence, evidently,
        $y\in c^m+\varrho B_{\ell^{\infty}}$.
    \end{enumerate}
     Altogether, we showed that $y \in c^{k} + \varrho B_{\ell^{\infty}}$ for some $k \in \{-m, \dots, m \}$, and \eqref{eq02:THM} follows.
     Therefore, $\alpha(I) \leq \frac{\sigma_{\ell}}{2} < \|I\|$, and $I$ is not maximally non-compact.
\end{proof}

\bigskip

Our next result concerns the case in which the target space is not a subspace of $c_0$, and therefore Theorem~\ref{THM:MN-c_0 subspaces} cannot be used, nevertheless, the embedding is maximally non-compact.

\begin{theorem}\label{ex: c_0 -> l_infty}
Let $I$ be the embedding operator $I\colon c_{0} \rightarrow \ell^{\infty} $. 
Then $I$ is maximally non-compact.
\end{theorem}

\begin{proof}
Let $x = \{x_{n}\}_{n=1}^{\infty} \in c_{0}$, then
$\|I(x)\| = \|x\|_{\infty}$. So, $I$ is a linear isometry from $c_{0}$ to $\ell^{\infty}$, which implies that $\|I\| = 1$, and it attains its norm at any $x\in B_{c_{0}}$ with $\|x\|_{\infty}=1$.
    
We proceed with the proof by contradiction.
Let us assume that $I$ is not maximally non-compact. Then
$\alpha(I) < \|I\| = 1$, and we take $\varrho \in (\alpha(I), 1)$.
Then, by the definition of the ball measure of non-compactness, we find $m \in \mathbb{N}$ and $y^{1},\dots, y^{m} \in \ell^{\infty}$ such that
$$I(B_{c_{0}}) \subset \bigcup_{i = 1}^m (y^{i} + \varrho B_{\ell^{\infty}}). $$

To produce a contradiction, we construct an element $a=\{a_k\}_{k=1}^\infty  \in B_{c_0}$ which is not contained in any of these $m$ balls. We define the sequence $a$ as
\begin{equation*}
    a_k = \begin{cases} 
    1, & \text{if } k \le m \text{ and } (y^k)_k < 0, \\
    -1, & \text{if } k \le m \text{ and } (y^k)_k \geq 0, \\
    0, & \text{if } k > m.
    \end{cases}
\end{equation*}
Since $a$ has only finitely many non-zero terms, $a \in c_0$. Furthermore, $\|a\|_\infty =  \sup_{n \in \mathbb{N}}|a_{n}|= 1$, so $a \in B_{c_0}$.

We now show that $a$ is excluded from every ball in the finite cover. Fix an arbitrary index $i \in \{1, \dots, m\}$. To check if $a$ is in the $i$-th ball, we look at the explicit coordinate $j(i) = i$.
\begin{enumerate}[a)]
    \item If $(y^i)_i < 0$, then $a_{j(i)} = a_i= 1$. Thus, $|a_i - (y^i)_i| = 1 - (y^i)_i > 1 > \varrho$.
    \item If $(y^i)_i \geq 0$, then $a_{j(i)} =a_i = -1$. Thus, $|a_i - (y^i)_i| = |-1 - (y^i)_i| = 1 + (y^i)_i \ge 1 > \varrho$.
\end{enumerate}
In both cases, we have $|a_i - (y^i)_i| > \varrho$. It follows that
$$\|a - y^i\|_\infty = \sup_{k \in \mathbb{N}} |a_k - (y^i)_k| \ge |a_i - (y^i)_i| > \varrho.$$
Since this holds for each $i \in \{1, \dots, m\}$, we conclude that $a \notin \bigcup_{i=1}^m (y^i + \varrho B_{\ell^\infty})$. This contradicts our initial covering assumption. Therefore, $I$ must be maximally non-compact.
\end{proof}

\begin{remark}
    Note that we cannot apply Theorem~\ref{THM-not max noncompact} to the embedding $I\colon c_0 \rightarrow \ell^{\infty}$ in Theorem~\ref{ex: c_0 -> l_infty}, because $\sigma_{c_0} = 2$. It is attained, for example, on the sequence $y =(1,-1, 0,0, \dots) \in B_{c_{0}}$. 
\end{remark}

\section{Lorentz sequence spaces}\label{S:Lorentz}
In this section we shall focus on one particular class of sequence spaces, namely the two-parameter scale of Lorentz spaces $\ell^{p,q}$. 
This class of spaces is surprisingly rich and contains many diverse examples of interesting relations. 

To begin, we single out those Lorentz spaces, which are contained in $c_0$.

\begin{lemma}\label{lorentz subset c_0} If $p,q \in(0,\infty]$ and $\min\{p,q\} < \infty$, then $\ell^{p,q} \subset c_{0}$. 
\end{lemma}

\begin{proof} Let us assume that $\ell^{p,q} \not\subset c_{0}$ and take $a \in \ell^{p,q} \setminus c_{0}$.
Then there exists $ \varepsilon > 0$ such that, for all $ n_{0}\in \mathbb{N} $, there exists $ n \geq n_{0}$ satisfying $|a_{n}| \geq \varepsilon$, which implies that $a_{n}^{*}   \geq  \varepsilon$ for every $n \in \mathbb{N}$.
First, let us look at the case when $q<\infty$ and $p\in(0,\infty]$. Then $$\infty > \|a\|_{p,q} \geq \varepsilon \left( \sum_{n=1}^{\infty}(n^{\frac{1}{p} - \frac{1}{q}})^{q}\right)^{\frac{1}{q}} = \varepsilon\left( \sum_{n=1}^{\infty}n^{\frac{q}{p}-1}\right)^{\frac{1}{q}} = \infty,$$
because, for all $p\in (0, \infty]$, one has $ (\frac{q}{p}-1) \geq -1$, and we know that
$\sum_{n=1}^{\infty}n^{\alpha}$ diverges for $\alpha \geq -1$. 
%Therefore, $(\varepsilon n^{\frac{1}{p} - \frac{1}{q}})_{n=1}^{\infty} \notin \ell^{q} $ and $ a\notin\ell^{p,q}$. 
This gives us a contradiction with our assumption that $a \in \ell^{p,q}$.
Similarly, for $q = \infty$ and $p\in(0,\infty)$, we get $\infty > \|a\|_{p,\infty} \geq \sup_{n\in \mathbb{N}} \{\varepsilon n^{1/p}\} = \infty$, which is a contradiction once again.
\end{proof}

The statements of the following two lemmas occur in \cite{Kato76}, however, the proofs of the statements are only sketched there. Some ideas of proofs are furthermore similar to analogous arguments known from the nonatomic case (see e.g.~\cite[Section~8.2]{Pic:13}), but in the discrete case other techniques have to be used.

The verification of the next lemma (restricted to cases $p,q\in[1,\infty)$) can be found in \cite[Lemma 2, p.~77]{Kato76}. We include the extended statement and its elementary proof for the sake of completeness.

\begin{lemma}\label{lemma(a_n)*ineq} Let $p\in(0,\infty]$, $q \in(0,\infty)$ and $a = \{a_{n}\}_{n=1}^{\infty} \in \ell^{p,q}$, then for all $n \in \mathbb{N}$ it holds that
\begin{enumerate}[1.]
    \item $a_{n}^{*} \leq \left( \frac{q}{p}\right)^{\frac{1}{q}}n^{-\frac{1}{p}}\| a \|_{p,q}$, if $0 < p \leq q < \infty$,
    \item $a_{n}^{*} \leq n^{-\frac{1}{p}}\| a \|_{p,q}$, if $0< q < p \leq \infty$.
\end{enumerate}   
\end{lemma}

\begin{proof} We will prove the individual cases.
\begin{enumerate}[1.] 
\item Let $0< p \leq q < \infty$ and fix $n \in \mathbb{N}$. Then

\begin{equation*}
 \| a \|_{p,q}^{q} 
 =\displaystyle\sum_{i=1}^{\infty} i^{\frac{q}{p} - 1}(a_{i}^{*})^{q}
 \geq \displaystyle\sum_{i=1}^{n} i^{\frac{q}{p} - 1}(a_{i}^{*})^{q}
\geq (a_{n}^{*})^{q}\displaystyle\sum_{i=1}^{n} i^{\frac{q}{p} - 1},
\end{equation*}
where the second inequality follows from the fact that $a_{i}^{*}$ is decreasing.
By the Lagrange mean value theorem, for each $i \in \mathbb{N}$ there exists $\xi \in (i-1, i)$ such that the function $f(t) = \frac{p}{q}t^{\frac{q}{p}}$ satisfies the equality
$$f^{'}(\xi) = \xi^{\frac{q}{p} - 1} = \frac{p}{q}\left[i^{\frac{q}{p}}-(i-1)^{\frac{q}{p}}\right].$$
Also $i^{\frac{q}{p} - 1} \geq \xi^{\frac{q}{p} - 1}$, because $\frac{q}{p} \geq 1$ and $i> \xi > 0$. Thus,
\begin{equation*}
    \| a \|_{p,q}^{q} 
    \geq (a_{n}^{*})^{q} \left(\frac{p}{q}\right) \sum_{i=1}^{n} \left( i^{\frac{q}{p}}-(i-1)^{\frac{q}{p}}\right) = (a_{n}^{*})^{q} \left( \frac{p}{q}\right) n^{\frac{q}{p}},
\end{equation*}
which yields the desired inequality.

%\begin{flalign*} \| a \|_{p,q}^{q}  & =\displaystyle\sum_{i=1}^{\infty} i^{\frac{q}{p} - 1}(a_{i}^{*})^{q}\\& \geq \displaystyle\sum_{i=1}^{n} i^{\frac{q}{p} - 1}(a_{i}^{*})^{q}\\& \geq (a_{n}^{*})^{q}\displaystyle\sum_{i=1}^{n} i^{\frac{q}{p} - 1}\\& \geq (a_{n}^{*})^{q}\frac{p}{q}\displaystyle\sum_{i=1}^{n}\left( i^{\frac{q}{p}}-(i-1)^{\frac{q}{p}}\right)\\& = (a_{n}^{*})^{q}\left( \frac{p}{q}\right)n^{\frac{q}{p}}. \end{flalign*} Here, the second inequality follows from the fact that $(a_{i}^{*})$ is the decreasing rearrangement. The third inequality results from the Lagrange mean value theorem. Indeed, for all $i\in \mathbb{N}$, there exists $ \xi \in (i-1, i)$, that satisfies for the function $f(t) =\frac{p}{q}t^{\frac{q}{p}}$, where $ t\geq 0$ following equality: $$f^{'}(\xi) = \xi^{\frac{q}{p} - 1} = \frac{p}{q}\left[i^{\frac{q}{p}}-(i-1)^{\frac{q}{p}}\right].$$ Also $i^{\frac{q}{p} - 1} \geq \xi^{\frac{q}{p} - 1}$, because $\frac{q}{p} \geq 1$ and $i> \xi > 0$. From the estimates above we get the third inequality and 1. is proven.
\item The second assertion follows analogously. We have
\begin{equation*}
    \| a \|_{p,q}^{q} 
    \geq (a_{n}^{*})^{q}\displaystyle\sum_{i=1}^{n} i^{\frac{q}{p} - 1}
    \geq (a_{n}^{*})^{q}n n^{\frac{q}{p} - 1}
    = (a_{n}^{*})^{q}n^{\frac{q}{p}}.
\end{equation*}
The second inequality follows from the fact that 
$\frac{q}{p} \in [0,1)$ and $i\leq n$. For the case $p = \infty$, we apply the convention $\frac{1}{\infty} = 0$. Therefore, $\left(\frac{q}{p} -1\right) = -1 <0$, so we get $ n^{\left(\frac{q}{p} -1\right)} \leq i^{\left(\frac{q}{p} -1\right)}$.
This completes the proof.

\end{enumerate}
\end{proof}

We can now see that Lemma~\ref{lorentz subset c_0} is an easy corollary of Lemma~\ref{lemma(a_n)*ineq} for parameters $p\in(0,\infty)$, $q \in(0,\infty)$. However, the proof of Lemma~\ref{lorentz subset c_0} gives us the result $\ell^{p,q} \subset c_{0}$ in a more straightforward way and also covers the remaining cases, namely $0<p<q =\infty$ and $p=\infty$, $q \in(0,\infty)$.

The next theorem can be found in \cite[Proposition 2(i), p.~77-78]{Kato76} for the parameter range $p,q \geq 1$. We extend it to the case when $p,q > 0$.

\begin{theorem}\label{lemma-inclusion-p,p,q1,q2} Let $0< p \leq \infty$, $0< q_{1} < q_{2}  \leq \infty$. Then $$\ell^{p,q_{1}} \hookrightarrow \ell^{p,q_{2}},$$ and, for all $ a \in \ell^{p,q_{1}}$, we have
    \begin{enumerate}[1.]
    \item $\|a \|_{p,q_{2}}\leq \left( \frac{q_{1}}{p}\right)^{\frac{1}{q_{1}} - \frac{1}{q_{2}}} \| a \|_{p,q_{1}}$ if $p < q_{1}$,
    \item $\| a \|_{p,q_{2}}\leq \| a \|_{p,q_{1}}$ if $p\geq q_{1}$.
    \end{enumerate}
\end{theorem}

\begin{proof} We will treat each case separately.
\begin{enumerate}[1.]
\item Let $p<q_{1}$ and $q_2\in (q_1,  \infty]$. We compute
    \begin{flalign*}
    \| a \|_{p,q_{2}} 
    & = \left\Vert n^{\frac{1}{p}-\frac{1}{q_2}} (a_{n}^{*})^{1-\frac{q_1}{q_2}}(a_{n}^{*})^{\frac{q_1}{q_2}}\right\Vert_{q_2}\\
    %& \leq  \left\Vert \left(\frac{q_{1}}{p}\right)^{\frac{1}{q_{1}} -\frac{1}{q_{2}}}\| a \|_{p,q_{1}}^{1-\frac{q_{1}}{q_{2}}}n^{-\frac{1}{p}+\frac{q_1}{pq_{2}}}n^{\frac{1}{p}-\frac{1}{q_2}} (a_{n}^{*})^{\frac{q_1}{q_2}}\right\Vert_{q_2}\\
    & \leq \left(\frac{q_{1}}{p}\right)^{\frac{1}{q_{1}} -\frac{1}{q_{2}}}\| a \|_{p,q_{1}}^{1-\frac{q_{1}}{q_{2}}}\left\Vert n^{\frac{q_1}{pq_2}-\frac{1}{q_2}} (a_{n}^{*})^{\frac{q_1}{q_2}}\right\Vert_{q_2}\\
     & = \left(\frac{q_{1}}{p}\right)^{\frac{1}{q_{1}} -\frac{1}{q_{2}}}\| a \|_{p,q_{1}}^{1-\frac{q_{1}}{q_{2}}}\left\Vert n^{\frac{1}{p}-\frac{1}{q_1}} a_{n}^{*}\right\Vert_{q_1}^{\frac{q_1}{q_2}}\\
     & = \left(\frac{q_{1}}{p}\right)^{\frac{1}{q_{1}} -\frac{1}{q_{2}}}\| a \|_{p,q_{1}}.
    \end{flalign*}
    Here, the first inequality follows again from the first part of Lemma~\ref{lemma(a_n)*ineq}. For the case $q_2 = \infty$, we apply the convention $\frac{1}{q_{2}}= \frac{1}{\infty} = 0$.

\item Now, let $p\geq q_{1}$ and $q_2\in (q_1,  \infty]$. Then, by the second part of Lemma~\ref{lemma(a_n)*ineq} and computations analogous to the first case, one has
\begin{equation*}
    \| a \|_{p,q_{2}}  =
    \left\Vert n^{\frac{1}{p}-\frac{1}{q_2}} (a_{n}^{*})^{1-\frac{q_1}{q_2}}(a_{n}^{*})^{\frac{q_1}{q_2}}\right\Vert_{q_2}
     \leq \| a \|_{p,q_{1}}^{1-\frac{q_{1}}{q_{2}}}\left\Vert n^{\frac{q_1}{pq_2}-\frac{1}{q_2}} (a_{n}^{*})^{\frac{q_1}{q_2}}\right\Vert_{q_2}
     = \| a \|_{p,q_{1}}.
\end{equation*}
\end{enumerate}
\end{proof}

We will now concentrate on establishing the embeddings $\ell^{p_{1},q_{1}} \hookrightarrow \ell^{p_{2},q_{2}}$ between the Lorentz sequence spaces based on the values of $p_1, p_{2}$ and $q_{1}, q_{2}$.
Then, we will try to calculate the exact values of the operator norms of these embeddings.

The next theorem is an extension of \cite[Proposition 2 (ii), p.~77-78]{Kato76} to the parameter range $p_1, q_1, p_2,q_2 > 0$. It provides a complete overview of the upper bounds for the norm of the embedding operator between Lorentz spaces, which are not all sharp (see Remark~\ref{Remark-open problem} for more).

\begin{theorem}\label{lemma-inclusions-pq1,pq2} 
Let $0< p_{1} < p_{2} \leq \infty$ and $q_{1}, q_{2} \in (0,\infty]$. Then $$\ell^{p_{1},q_{1}} \hookrightarrow \ell^{p_{2},q_{2}},$$ 
and, for all $ a \in \ell^{p_{1},q_{1}}$, we have
\begin{enumerate}[1.]
    \item $\| a \|_{p_{2},q_{2}}\leq  \| a \|_{p_{1},q_{1}}$,  if $q_{1} \leq p_{1}$ and $q_{1} \leq q_{2}$,
    \item $\| a \|_{p_{2},q_{2}}\leq \left( \displaystyle\sum_{n=1}^{\infty} n^{\frac{q_{2}}{p_{2}} -\frac{q_{2}}{p_{1}}- 1}\right)^{\frac{1}{q_{2}}}\|a \|_{p_{1},\infty}$, if $p_{1} < q_{1} = \infty$ and $q_{2}<\infty$,
    \item $\| a \|_{p_{2},q_{2}}\leq \left( \frac{q_{1}}{p_{1}}\right)^{\frac{1}{q_{1}} - \frac{1}{q_{2}}}\| a \|_{p_{1},q_{1}}$, if $p_{1} < q_{1} < q_{2}\leq\infty$,
    \item $\| a \|_{p_{2},q_{2}}\leq \left( \displaystyle\sum_{n=1}^{\infty} n^{\frac{q_{2}}{p_{2}} -\frac{q_{2}}{p_{1}}- 1}\right)^{\frac{1}{q_{2}}} \max\left(1, \left(\frac{q_1}{p_1}\right)^{\frac{1}{q_1}}\right) \| a \|_{p_{1},q_{1}}$, if $q_{2} < q_{1} < \infty$,
    \item $\| a \|_{p_{2},q_{2}}\leq \| a \|_{p_{1},q_{1}}$, if $p_{1} < q_{1}  = q_{2} \leq \infty$.
    
    \end{enumerate}
\end{theorem}

\begin{proof}
We will prove the individual cases.
\begin{enumerate}[1.]
\item Let $q_{1} \leq p_{1} < \infty$ and $q_{1} \leq q_{2} \leq \infty$. Then
    \begin{flalign*}
     \| a \|_{p_2,q_{2}} & =
    \left\Vert n^{\frac{1}{p_2}-\frac{1}{q_2}} (a_{n}^{*})^{1-\frac{q_1}{q_2}}(a_{n}^{*})^{\frac{q_1}{q_2}}\right\Vert_{q_2} \leq \| a \|_{p_1,q_{1}}^{1-\frac{q_{1}}{q_{2}}}\left\Vert n^{\frac{1}{p_2}-\frac{1}{p_1}}n^{\frac{q_1}{p_1q_2}-\frac{1}{q_2}} (a_{n}^{*})^{\frac{q_1}{q_2}}\right\Vert_{q_2}\\
    & \leq \| a \|_{p_1,q_{1}}^{1-\frac{q_{1}}{q_{2}}}\left\Vert n^{\frac{q_1}{p_1q_2}-\frac{1}{q_2}} (a_{n}^{*})^{\frac{q_1}{q_2}}\right\Vert_{q_2} = \| a \|_{p_1,q_{1}}^{1-\frac{q_{1}}{q_{2}}}\left\Vert n^{\frac{1}{p_1}-\frac{1}{q_1}} a_{n}^{*}\right\Vert_{q_1}^{\frac{q_1}{q_2}}\\
    & = \| a \|_{p_1,q_{1}}.
    \end{flalign*}
Here, the first inequality follows from the second part of Lemma~\ref{lemma(a_n)*ineq}, namely
$$a_{n}^{*} \leq n^{-\frac{1}{p_{1}}}\| a \|_{p_{1},q_{1}},$$
and we get the second inequality from the fact that
$$n^{\frac{1}{p_{2}} -\frac{1}{p_{1}}}\leq1.$$
We note that the argument works also in the case when $q_2=\infty$ under the convention that $\frac{1}{q_{2}} = \frac{1}{\infty}= 0$ and $a^{\frac{1}{\infty}}=1$ for any nonnegative quantity $a$.

\item Let $p_{1} < q_{1} = \infty$ and $q_{2} < \infty$. Then
    \begin{equation*}
    \|a \|_{p_{2},q_{2}}^{q_{2}}
    = \displaystyle\sum_{n=1}^{\infty} n^{\frac{q_{2}}{p_{2}} -\frac{q_{2}}{p_{1}}-1}(a_{n}^{*})^{q_{2}}n^{\frac{q_{2}}{p_{1}}} \leq\left(\displaystyle\sup_{n\in\mathbb{N}}a_{n}^{*}n^{\frac{1}{p_{1}}}\right)^{q_{2}} \displaystyle\sum_{n=1}^{\infty} n^{\frac{q_{2}}{p_{2}} -\frac{q_{2}}{p_{1}}-1}
    =\|a \|_{p_{1},\infty}^{q_{2}}\displaystyle\sum_{n=1}^{\infty} n^{\frac{q_{2}}{p_{2}} -\frac{q_{2}}{p_{1}} -1},
    \end{equation*}
in which the sum $\displaystyle\sum_{n=1}^{\infty} n^{\frac{q_{2}}{p_{2}} -\frac{q_{2}}{p_{1}} -1}$ converges because $\frac{q_{2}}{p_{2}} -\frac{q_{2}}{p_{1}} -1 <-1$.

\item Let $p_{1} < q_{1} < \infty$ and $q_1 < q_2 \leq \infty$. Then  
    \begin{flalign*}
     \| a \|_{p_2,q_{2}} & =
    \left\Vert n^{\frac{1}{p_2}-\frac{1}{q_2}} (a_{n}^{*})^{1-\frac{q_1}{q_2}}(a_{n}^{*})^{\frac{q_1}{q_2}}\right\Vert_{q_2} \leq \left( \frac{q_{1}}{p_{1}}\right)^{\frac{1}{q_{1}} - \frac{1}{q_{2}}}\| a \|_{p_1,q_{1}}^{1-\frac{q_{1}}{q_{2}}}\left\Vert n^{\frac{1}{p_2}-\frac{1}{p_1}}n^{\frac{q_1}{p_1q_2}-\frac{1}{q_2}} (a_{n}^{*})^{\frac{q_1}{q_2}}\right\Vert_{q_2}\\
    & \leq \left( \frac{q_{1}}{p_{1}}\right)^{\frac{1}{q_{1}} - \frac{1}{q_{2}}}\| a \|_{p_1,q_{1}}^{1-\frac{q_{1}}{q_{2}}}\left\Vert n^{\frac{q_1}{p_1q_2}-\frac{1}{q_2}} (a_{n}^{*})^{\frac{q_1}{q_2}}\right\Vert_{q_2} = \left( \frac{q_{1}}{p_{1}}\right)^{\frac{1}{q_{1}} - \frac{1}{q_{2}}}\| a \|_{p_1,q_{1}}^{1-\frac{q_{1}}{q_{2}}}\left\Vert n^{\frac{1}{p_1}-\frac{1}{q_1}} a_{n}^{*}\right\Vert_{q_1}^{\frac{q_1}{q_2}}\\
    & = \left( \frac{q_{1}}{p_{1}}\right)^{\frac{1}{q_{1}} - \frac{1}{q_{2}}}\| a \|_{p_1,q_{1}}.
    \end{flalign*}
    Here, the first inequality follows from the first part of Lemma~\ref{lemma(a_n)*ineq}, namely
    $$a_{n}^{*} \leq \left( \frac{q_{1}}{p_{1}}\right)^{\frac{1}{q_{1}}}n^{-\frac{1}{p_{1}}}\| a \|_{p_{1},q_{1}},$$ and the second inequality from the fact that $$n^{\frac{1}{p_{2}} -\frac{1}{p_{1}}} \leq 1.$$
    If $q_2 = \infty$, we note that the argument also works under the convention that $\frac{1}{q_{2}} = \frac{1}{\infty} = 0$ and $a^{\frac{1}{\infty}}=1$ for any nonnegative quantity $a$.
    
\item Let $q_{2} < q_{1} < \infty$. Using the embedding $\ell^{p_1, q_1} \hookrightarrow \ell^{p_1, \infty}$ established by Theorem~\ref{lemma-inclusion-p,p,q1,q2}, we know that $\|a\|_{p_1, \infty} \leq \max\left(1, \left(\frac{q_1}{p_1}\right)^{\frac{1}{q_1}}\right) \|a\|_{p_1, q_1}$. Combining this with the technique from part 2 of this proof, we obtain
    \begin{flalign*}
     \|a \|_{p_{2},q_{2}}
     & = \left( \displaystyle\sum_{n=1}^{\infty} n^{\frac{q_{2}}{p_{2}} - 1}(a_{n}^{*})^{q_{2}} \right)^{\frac{1}{q_2}} 
     \leq \|a\|_{p_1, \infty} \left( \displaystyle\sum_{n=1}^{\infty} n^{\frac{q_{2}}{p_{2}} -\frac{q_{2}}{p_{1}}- 1} \right)^{\frac{1}{q_2}} \\
     & \leq \left( \displaystyle\sum_{n=1}^{\infty} n^{\frac{q_{2}}{p_{2}} -\frac{q_{2}}{p_{1}}- 1} \right)^{\frac{1}{q_2}} \max\left(1, \left(\frac{q_1}{p_1}\right)^{\frac{1}{q_1}}\right) \| a \|_{p_{1},q_{1}}.
    \end{flalign*}

\item Let $p_{1} < q_{1} = q_{2} \leq \infty$. Since $p_1 < p_2$, we have $\frac{1}{p_2} < \frac{1}{p_1}$, which implies $n^{\frac{1}{p_2} - \frac{1}{q_2}} \le n^{\frac{1}{p_1} - \frac{1}{q_1}}$ for all $n \in \mathbb{N}$. Therefore,
    \begin{flalign*}
    \| a \|_{p_2,q_{2}} 
    & = \left\Vert n^{\frac{1}{p_2}-\frac{1}{q_2}} a_{n}^{*} \right\Vert_{q_2} \\
    & \leq \left\Vert n^{\frac{1}{p_1}-\frac{1}{q_1}} a_{n}^{*} \right\Vert_{q_1} \\
    & = \| a \|_{p_1,q_{1}}.
    \end{flalign*}
    If $q_{1} = q_{2} = \infty$, we note that the argument holds identically under the standard convention $\frac{1}{\infty} = 0$.
\end{enumerate}
\end{proof}

The following two theorems are new results which establish sharpness of the relevant upper bounds of the identity operator obtained in Theorem~\ref{lemma-inclusions-pq1,pq2} and of Kato’s previous results, namely \cite[Proposition 2, p.~77-78]{Kato76}. 

\begin{theorem}\label{lemma ||I||-p12,q12} 
Let $0< p_{1} \leq p_{2} \leq \infty$ and $q_{1},q_{2} \in (0,\infty]$. If $p_1 = p_2$, let us also assume $q_1 <q_2$.
Then the norm of the embedding operator $I\colon \ell^{p_{1},q_{1}} \rightarrow \ell^{p_{2},q_{2}}$ satisfies
\begin{enumerate}[1.]
    \item $\|I\| = 1$, if $q_{1} \leq p_{1}$ and $q_{1} \leq q_{2}$,
    \item $\|I\| =  1$, if $p_{1} < q_{1} = q_{2} \leq \infty$,
    \item $\|I\| = \left( \displaystyle\sum_{n=1}^{\infty} n^{\frac{q_{2}}{p_{2}} -\frac{q_{2}}{p_{1}}- 1}\right)^{\frac{1}{q_{2}}}$, if $p_{1} < q_{1} = \infty$, $q_{2} < \infty$ and $p_{1} < p_{2}$.
   \end{enumerate}
\end{theorem}

\begin{proof}We will prove the individual cases separately.
\begin{enumerate}[1.]
\item Let $q_{1} \leq p_{1}$ and $q_1 \leq q_2$. We get the upper estimate of the operator norm for $p_{1} < p_{2}$ from the first part of Theorem~\ref{lemma-inclusions-pq1,pq2} and for $p_{1} = p_{2}$ from the second part of Theorem~\ref{lemma-inclusion-p,p,q1,q2}. So $\|I\| \leq 1$.
 We will now show that $I$ attains $1$ on the closed unit ball.
 For arbitrary $e^{j}$, we have $(e^{j})^{*} = e^{1}$ and
 $$\|e^{j}\|_{p_{1},q_{1}} =\left(\displaystyle\sum_{n=1}^{\infty} n^{\frac{q_{1}}{p_{1}} - 1}((e_{n}^{j})^{*})^{q_{1}}\right)^{\frac{1}{q_{1}}} = \left(1^{\frac{q_{1}}{p_{1}} - 1}1^{q_{1}}\right)^{\frac{1}{q_{1}}} =  1.$$ 
 Therefore, $e^{j} \in B_{\ell^{p_{1},q_{1}}}$. Now, we need to show that also $\|I(e^{j})\|_{p_{2},q_{2}} = \|e^{j}\|_{p_{2},q_{2}} = 1$.
\begin{enumerate}[a)]
 \item Let $q_{2} < \infty$. Then $\|e^{j}\|_{p_{2},q_{2}} = 1$ follows from the computations above.
 \item Let $q_{2} = \infty$. Then $\|e^{j}\|_{p_{2},q_{2}} = \displaystyle\sup_{n\in\mathbb{N}}((e_{n}^{j})^{*})n^{\frac{1}{p_{2}}} = 1$.
 \end{enumerate}
 It already follows that $\|I\| = 1$.

\item Let $p_{1} < q_{1} = q_{2} \leq \infty$. We get the upper estimate $\|I\| \leq 1$ from the fifth part of Theorem~\ref{lemma-inclusions-pq1,pq2}. We can show that $I$ attains $1$ on the closed unit ball at an arbitrary $e^{j}$. The computations are similar to those in the proof of 1, more precisely, if $q_1 < \infty$,
$$\|e^{j}\|_{p_{1},q_{1}} = 1 = \|e^{j}\|_{p_{2},q_{2}},$$
and if $q_1 = \infty$,
$$\|e^{j}\|_{p_{1},\infty} = \displaystyle\sup_{n\in\mathbb{N}}((e_{n}^{j})^{*})n^{\frac{1}{p_{1}}} = 1 = \|e^{j}\|_{p_{2},\infty}.$$
Thus, $\|I\| = 1$.

\item Let $p_{1} < q_{1} = \infty$ and $q_{2} < \infty$. Once again, we obtain the upper estimate for the operator norm from the second part of Theorem~\ref{lemma-inclusions-pq1,pq2}.
Therefore,
$$
\|I\| \leq \left( \displaystyle\sum_{n=1}^{\infty} n^{\frac{q_{2}}{p_{2}} -\frac{q_{2}}{p_{1}}- 1}\right)^{\frac{1}{q_{2}}}.
$$ Now, we will show that $I$ attains this upper estimate on the closed unit ball.
Let us consider the sequence $ a = a^{*} = (1, 2^{-\frac{1}{p_{1}}}, 3^{-\frac{1}{p_{1}}},\dots, n^{-\frac{1}{p_{1}}}, \dots )$. Then,
$$\|a\|_{p_{1},\infty} = \displaystyle\sup_{n\in\mathbb{N}}a_{n}^{*}n^{\frac{1}{p_{1}}} = \displaystyle\sup_{n\in\mathbb{N}}n^{-\frac{1}{p_{1}}}n^{\frac{1}{p_{1}}} = 1.$$
Hence $a\in B_{\ell^{p_{1},\infty}}$, and we will now compute the norm of $\|I(a)\|_{p_{2},q_{2}} = ||a||_{p_{2},q_{2}}$. One has
\begin{equation*}
\|a\|_{p_{2},q_{2}} 
=\left(\displaystyle\sum_{n=1}^{\infty} n^{\frac{q_{2}}{p_{2}} - 1}(a_{n}^{*})^{q_{2}}\right)^{\frac{1}{q_{2}}} 
= \left(\displaystyle\sum_{n=1}^{\infty} n^{\frac{q_{2}}{p_{2}} - 1}(n^{-\frac{1}{p_{1}}})^{q_{2}}\right)^{\frac{1}{q_{2}}} 
= \left(\displaystyle\sum_{n=1}^{\infty} n^{\frac{q_{2}}{p_{2}} -\frac{q_{2}}{p_{1}}- 1}\right)^{\frac{1}{q_{2}}}.
\end{equation*}
So, indeed, $I$ attains its upper estimate on the closed unit ball, and the case 3 is proved. This completes the proof of the theorem.
\end{enumerate}

\end{proof}

\begin{remark}\label{Remark-open problem}
In cases 3 and 4 of Theorem~\ref{lemma-inclusions-pq1,pq2}, as well as in the first part of Theorem~\ref{lemma-inclusion-p,p,q1,q2} when $q_2 < \infty$, we do not know whether the established upper bounds for the operator norms are optimal in full generality. The explicit parameter ranges in question are 
\begin{itemize}
    \item $0 < p < q_{1} < q_{2} < \infty$ (corresponding to the first part of Theorem~\ref{lemma-inclusion-p,p,q1,q2}),
    \item $0 < p_{1} < p_{2} \leq \infty$ and $p_{1} < q_{1} < q_{2} \leq \infty$ (corresponding to the third part of Theorem~\ref{lemma-inclusions-pq1,pq2}),
    \item $0 < p_{1} < p_{2} \leq \infty$ and $q_{2} < q_{1} < \infty$ (corresponding to the fourth part of Theorem~\ref{lemma-inclusions-pq1,pq2}).
\end{itemize}
Determining the exact operator norm $\|I\|$ of the embeddings $I\colon \ell^{p_{1},q_{1}} \rightarrow \ell^{p_{2},q_{2}}$ for these specific cases remains an open question.
\end{remark}

A partial sharp result in this direction is offered in the next theorem.

\begin{theorem}\label{Lemma:Riemann}Let $0 < p < q <\infty$. Then the norm of the embedding operator $I\colon \ell^{p,q} \rightarrow \ell^{p,\infty}$ satisfies $$\|I\| = \left( \frac{q}{p}\right)^{\frac{1}{q}}.$$  
\end{theorem}
\begin{proof}
We obtain the estimate
$\|I\| \leq \left( \frac{q}{p}\right)^{\frac{1}{q}}$ from the first part of Theorem~\ref{lemma-inclusion-p,p,q1,q2}.\\
We want to show that there exists a sequence $\{a^{n}\}_{n=1}^{\infty}$ of nonzero elements, where $a^{n}\in \ell^{p,q}$ for all $n \in \mathbb{N}$, such that
\begin{equation}\label{eq:lim-rieman}
    \displaystyle\lim_{n\rightarrow \infty}\frac{\|a^{n}\|_{p,\infty}}{\|a^{n}\|_{p,q}} =\left(\frac{q}{p}\right)^{\frac{1}{q}}.
\end{equation}
For every $n\in \mathbb{N}$, set $a^{n} = (1, 1, 1,\dots, 1, 0, 0\dots)$, where $a_{i}^{n} =  1$ for $i \in \{1,2,\dots,n\}$ and
$a_{i}^{n} = 0 $ for $i > n$. Then, for every $n\in \mathbb{N}$, we have
$$ \|a^{n}\|_{p,\infty} = \sup_{i \leq n} i^{\frac{1}{p}} = n^{\frac{1}{p}},
$$
while
$$ \|a^{n}\|_{p,q} = \left(\sum_{i = 1}^{n} i^{\frac{q}{p} - 1} \right)^{\frac{1}{q}}.$$
Now, from the theory of the Riemann integral, we know that
%\begin{flalign*} \displaystyle\lim_{n\rightarrow \infty}\frac{\|a^{n}\|_{p,\infty}}{\|a^{n}\|_{p,q}} & =\displaystyle\lim_{n\rightarrow \infty}\frac{n^{\frac{1}{p}}}{\displaystyle\left(\sum_{i = 1}^{n} i^{\frac{q}{p} - 1} \right)^{\frac{1}{q}}}\\ & =\displaystyle\frac{1}{\displaystyle\left(\lim_{n\rightarrow \infty}\sum_{i = 1}^{n} \left(\frac{i}{n}\right)^{\frac{q}{p} - 1}\frac{1}{n} \right)^{\frac{1}{q}}}\\& = \displaystyle\left(\int_{0}^{1} x^{\frac{q}{p} - 1}dx \right)^{-\frac{1}{q}}\\ & = \left(\frac{q}{p}\right)^{\frac{1}{q}}.\end{flalign*}

\begin{equation*}
    \displaystyle\lim_{n\rightarrow \infty}\frac{\|a^{n}\|_{p,\infty}}{\|a^{n}\|_{p,q}} 
     =\displaystyle\lim_{n\rightarrow \infty}\frac{n^{\frac{1}{p}}}{\displaystyle\left(\sum_{i = 1}^{n} i^{\frac{q}{p} - 1} \right)^{\frac{1}{q}}}
     =\displaystyle\left(\lim_{n\rightarrow \infty}\sum_{i = 1}^{n} \left(\frac{i}{n}\right)^{\frac{q}{p} - 1}\frac{1}{n} \right)^{-\frac{1}{q}}
    = \displaystyle\left(\int_{0}^{1} x^{\frac{q}{p} - 1}dx \right)^{-\frac{1}{q}}\\
     = \left(\frac{q}{p}\right)^{\frac{1}{q}}.
\end{equation*}
Hence $\{a^{n}\}_{n=1}^{\infty}$ satisfies \eqref{eq:lim-rieman}, which establishes $\|I\| = \left( \frac{q}{p}\right)^{\frac{1}{q}}$.
\end{proof}

\section{Examples}\label{S:examples}

We shall now present examples of embeddings between Lorentz sequence spaces which illustrate Theorem ~\ref{THM:MN-c_0 subspaces}. 

\begin{proposition} 
\label{ex01:pq-I-c_0} Let $p,q \in(0,\infty]$ and $\min\{p,q\}<\infty$. Then the embedding operator $I\colon \ell^{p,q} \rightarrow c_{0}$ is maximally non-compact.
\end{proposition}
\begin{proof} We will verify the assumptions of Theorem~\ref{THM:MN-c_0 subspaces}. First,
\begin{enumerate}[1.]
    \item $\ell^{p,q} \subset c_{0}$ by Lemma~\ref{lorentz subset c_0}.
    \item We find the norm of the operator $I$. For some relations between $p$ and $q$, we already know that $\|I\| = 1$ from Theorem~\ref{lemma ||I||-p12,q12}, because $\|\cdot\|_{c_{0}} = \|\cdot\|_{\infty} = \|\cdot\|_{\infty,\infty}$.
    However, the direct computation of the norm covers all the admissible cases in a more straightforward way. Therefore, let us fix an arbitrary sequence $a = \{a_{n}\}_{n=1}^{\infty} \in \ell^{p,q}$.
    \begin{enumerate}[a)]
        \item  Let $q = \infty$. Then $\|I(a)\| = \|a\|_{\infty} = \sup_{n \in \mathbb{N}}|a_{n}| \leq \sup_{n \in \mathbb{N}}\{n^{\frac{1}{p}} a_{n}^{*}\} = \|a\|_{p,q}$, where the inequality results from the fact that $n^{\frac{1}{p}} \geq 1$ for all $ n \in \mathbb{N}$. 
        \item Let $q<\infty$. Then $\|I(a)\|=\|a\|_{\infty}=\sup_{n\in\mathbb{N}}|a_{n}| \leq \left(\sum_{n=1}^{\infty}(a_{n}^{*})^{q}n^{\frac{q}{p}-1}\right)^{\frac{1}{q}}=\|a\|_{p,q},$ where the inequality results from the fact that $$(\sup_{n\in\mathbb{N}}|a_{n}|)^{q}=(a_{1}^{*})^{q}=1^{\frac{q}{p}-1}(a_{1}^{*})^{q} \leq \sum_{n=1}^{\infty}(a_{n}^{*})^{q}n^{\frac{q}{p}-1}.$$
    \end{enumerate}
    Thus, $ \|I\| \leq 1$ for any $p,q \in(0,\infty]$ and $\min\{p,q\}<\infty$.
    Now, we will show that $I$ achieves $1$ on the closed unit ball. For an arbitrary $j\in\mathbb N$, we have $(e^{j})^{*} = e^{1}$. Therefore, $\|e^{j}\|_{p,q} = \|e^{1}\|_{p,q} = 1 $, and also $\|I(e^{j})\|_{\infty} = \|e^{j}\|_{\infty} = 1 $. Hence, $e^{j} \in B_{\ell^{p,q}}$, $\|I\| = 1$ and $\ell^{p,q} \hookrightarrow c_{0}$.
   \end{enumerate}
Thus, by Theorem~\ref{THM:MN-c_0 subspaces}, the embedding operator $I$ is maximally non-compact and $\alpha(I) = 1$.
\end{proof}

\begin{proposition} Let $0< p \leq \infty$, $0< q_{1} < q_{2}  \leq \infty$ and $\min\{p,q_{2}\}< \infty$. Then, the embedding operator $I\colon \ell^{p,q_{1}} \rightarrow \ell^{p,q_{2}}$ is maximally non-compact.
\end{proposition}
\begin{proof}
We have
\begin{enumerate}[1.]
    \item $\ell^{p,q_{1}} \hookrightarrow \ell^{p,q_{2}}$ from Theorem~\ref{lemma-inclusion-p,p,q1,q2},
    \item $\ell^{p,q_{1}}$, $\ell^{p,q_{2}} \subset c_{0}$ from Lemma~\ref{lorentz subset c_0}.
\end{enumerate}
So, the prerequisites of Theorem~\ref{THM:MN-c_0 subspaces} are satisfied, the embedding operator $I$  is maximally non-compact and $\alpha(I) =\|I\|$, where the values of $\|I\|$ for some individual relations between $p$, $q_{1}$ and $q_{2}$ are computed in Theorem~\ref{lemma ||I||-p12,q12} or Theorem~\ref{Lemma:Riemann}.
\end{proof}

\begin{proposition} Let $0 < p_{1} < p_{2} \leq \infty$, $q_{1}, q_{2} \in(0,\infty]$ and $\min\{p_{2},q_{2}\}< \infty$. Then the embedding operator $I\colon \ell^{p_{1},q_{1}} \rightarrow \ell^{p_{2},q_{2}}$ is maximally non-compact.
\end{proposition}
\begin{proof}
We have
\begin{enumerate}[1.]
    \item $\ell^{p_{1},q_{1}} \hookrightarrow \ell^{p_{2},q_{2}}$ and $0<\|I\| < \infty $ from Theorem~\ref{lemma-inclusions-pq1,pq2},
    \item and $\ell^{p_{1},q_{1}}$, $\ell^{p_{2},q_{2}} \subset c_{0}$ from Lemma~\ref{lorentz subset c_0}.
\end{enumerate}
So, the prerequisites of Theorem~\ref{THM:MN-c_0 subspaces} are satisfied, the embedding operator $I$  is maximally non-compact and $\alpha(I) =\|I\|$, where the values of $\|I\|$, for some individual relations between $p_{1}$, $q_{1}$ and $q_{2}$, are computed in Theorem~\ref{lemma ||I||-p12,q12}.
\end{proof}
We will now present examples of embeddings of Lorentz sequence spaces into $\ell^{\infty}$ which are not maximally non-compact. This will be shown by finding their span $\sigma_{\ell^{p,q}}$ and applying Theorem~\ref{THM-not max noncompact}. 

\begin{proposition}   
Let $p \in(0,\infty)$ and $q \in(0,\infty]$. Then the embedding operator $I\colon \ell^{p,q} \rightarrow \ell^{\infty}$, $\|I\| = 1$ is not maximally non-compact. More precisely,
\begin{enumerate}[1.]
    \item $\alpha(I) \leq 2^{-\frac{1}{q}} < 1$ when $1 \leq q < \infty$ and $p\leq q$,
    \item $\alpha(I) \leq 2^{-1} < 1$ when $0 < q < 1$ and $p\leq q$,
    \item $\alpha(I) \leq 2^{-1}(1+2^{-\frac{1}{p}}) < 1$ when $q = \infty$ and $p<q$,
    \item $\alpha(I) \leq 2^{-1}(1+2^{-\frac{1}{p}}) < 1$ when $q<p$.
\end{enumerate}
\end{proposition} 

\begin{proof}
    We have $\ell^{p,q} \subset c_{0} \subset \ell^{\infty}$ by Lemma~\ref{lorentz subset c_0}. Furthermore, $\|I\| = 1$ follows from the computations in Proposition~\ref{ex01:pq-I-c_0}. 
    \begin{enumerate}[1.]
    \item Let $1 \leq q < \infty$ and $p\leq q$. Denote $\sigma = 2^{1-\frac{1}{q}}$. Then, $\|I\| \leq \sigma < 2\|I\|$. Assume $y \in B_{\ell^{p,q}}$. Then $ y \in B_{\ell^\infty} $ since $\|I\|=1$, hence $ |y_{j}| \leq 1 \leq \sigma$. We claim that 
    
    \begin{equation}\label{eq03:pq-I-infty}
        \sup y - \inf y \leq \sigma.
    \end{equation}
    If $y=0$, the claim holds trivially. Assume $y \neq 0$. Given $\varepsilon > 0$, we find $s,i \in \mathbb{N}$ such that $y_{s} > \sup y - \varepsilon$ and $y_{i} < \inf y + \varepsilon$. We get
    \begin{flalign*}
        1 & \geq \|y\|_{p,q}
         = \left(\displaystyle\sum_{n=1}^{\infty} (y_{n}^{*})^{q} n^{\frac{q}{p} - 1}\right)^{\frac{1}{q}}
          \geq \left(\displaystyle\sum_{n=1}^{\infty} (y_{n}^{*})^{q}\right)^{\frac{1}{q}} = \left(\displaystyle\sum_{n=1}^{\infty} |y_{n}|^{q}\right)^{\frac{1}{q}}\\ 
        &\geq(|y_{s}|^{q}+|y_{i}|^{q})^{\frac{1}{q}}\geq 2^{\frac{1}{q}-1}(|y_{s}|+|y_{i}|)\\
        & > \frac{1}{\sigma}(\sup y - \inf y - 2\varepsilon),
    \end{flalign*}
    where the second inequality follows from our prerequisite $p\leq q$, hence $n^{\frac{q}{p} - 1} \geq 1$, and the fourth one follows from Lemma~\ref{a+b q>1}.
    Now, by letting $\varepsilon \rightarrow 0_{+}$, we get our claim. Therefore,
    $$
        \sigma_{\ell^{p,q}} \leq \sigma < 2\|I\|,
    $$ 
    and
    all the prerequisites of Theorem~\ref{THM-not max noncompact} are satisfied. So, $I$ is not maximally non-compact and $\alpha(I) \leq \frac{\sigma}{2} < 1 = \|I\|$.

    \item Let $0< q < 1$ and $p\leq q$. Denote $\sigma =1$. Then $\|I\| \leq \sigma < 2\|I\|$. Analogously to the previous case, we assume that $y \in B_{\ell^{p,q}}$. Then  $ |y_{j}| \leq 1 \leq \sigma$, and we claim that 
    \begin{equation*}
        \sup y - \inf y \leq 1.
    \end{equation*}
    If $y=0$, the claim holds trivially. Assume $y \neq 0$. Given $\varepsilon > 0$, we find $s,i \in \mathbb{N}$ such that $y_{s} > \sup y - \varepsilon$ and $y_{i} < \inf y + \varepsilon$. We get
    \begin{flalign*}
        1 & \geq \|y\|_{p,q} 
         = \left(\displaystyle\sum_{n=1}^{\infty} (y_{n}^{*})^{q} n^{\frac{q}{p} - 1}\right)^{\frac{1}{q}}\geq \left(\displaystyle\sum_{n=1}^{\infty} (y_{n}^{*})^{q}\right)^{\frac{1}{q}} = \left(\displaystyle\sum_{n=1}^{\infty} |y_{n}|^{q}\right)^{\frac{1}{q}}\\ 
        &\geq(|y_{s}|^{q}+|y_{i}|^{q})^{\frac{1}{q}} \geq |y_{s}|+|y_{i}|\\
        & > \sup y - \inf y - 2\varepsilon,
    \end{flalign*}
    where the second inequality follows from our prerequisite $p\leq q$, hence $n^{\frac{q}{p} - 1} \geq 1$, and the fourth one follows from Lemma~\ref{a+b q<1}. Now, by letting $\varepsilon \rightarrow 0_{+}$, we get our claim. Therefore, 
    $$
        \sigma_{\ell^{p,q}} \leq 1 < 2\|I\|,
    $$ 
    and
    all the prerequisites of Theorem~\ref{THM-not max noncompact} are satisfied. So, $I$ is not maximally non-compact and $\alpha(I) \leq \frac{1}{2} < 1 = \|I\|$.

    \item Now, let $q = \infty$.
    We will proceed with the proof similarly to the way we did in the first part. The only difference is that here we set 
    $\sigma = 1+2^{-\frac{1}{p}}$. Assume $y\in B_{\ell^{p,\infty}}$.
    We again need to verify the inequality 
        \begin{equation}\label{eq2-sup-inf}
            \sup y - \inf y \leq \sigma.
        \end{equation}
    Let $y\in B_{\ell^{p,\infty}}$, then
    \begin{equation}\label{eq04:pq-I-infty}
        \|y\|_{p,\infty} = \displaystyle\sup_{n\in \mathbb{N}} \{n^{1/p}y_{n}^{*}\} \leq 1 \text{, hence }  \displaystyle\sup_{n\in \mathbb{N}} \{y_{n}\} \leq 1.
    \end{equation}
    Inequality \eqref{eq04:pq-I-infty} implies that the decreasing rearrangement $y^{*}$ of $y$ must satisfy $y_{n}^{*}\leq n^{-\frac{1}{p}}$  for all $n \in \mathbb{N}$.
    So, every sequence from $B_{\ell^{p,\infty}}$ has its decreasing rearrangement bounded from above by the decreasing sequence $\{n^{-1/p}\}_{n=1}^{\infty}$ and from below by $\{-n^{-\frac{1}{p}}\}_{n=1}^{\infty}$. We shall now distinguish two cases.
    \begin{enumerate}[a)]
        \item If $\sup y \leq 2^{-\frac{1}{p}}$, then, owing to the observation above, $\inf y \geq -1$, and inequality \eqref{eq2-sup-inf} holds.
        \item If $\sup y > 2^{-\frac{1}{p}}$, then, once again, $y_{n}^{*}$ is bounded by the decreasing sequence $n^{-\frac{1}{p}}$, so there can be only one index $j \in \mathbb{N}$ such that $y_{j} > 2^{-\frac{1}{p}}$. Therefore, $y_{1}^{*} = y_{j}$
        and we get that for all $ i \neq j \colon |y_{i}| \leq 2^{-\frac{1}{p}}$, which implies that $\inf y \geq -2^{-\frac{1}{p}}$. Thus, $\sup y - \inf y \leq 1 - (-2^{-\frac{1}{p}}) = \sigma$ holds for all $y\in B_{\ell^{p,q}}$, and inequality \eqref{eq2-sup-inf} follows.
    \end{enumerate}

    We computed that $\sigma_{\ell^{p,q}} \leq \sigma = 1+2^{-\frac{1}{p}} < 2\|I\|$.
    Now, once again, by Theorem~\ref{THM-not max noncompact}, we get that $I$ is not maximally non-compact and $\alpha(I) \leq \varrho = 2^{-1}(1+2^{-\frac{1}{p}}) < 1$.

     \item Now, let $q<p$.
     By the second part of Theorem~\ref{lemma-inclusion-p,p,q1,q2}, $\ell^{p,q} \hookrightarrow \ell^{p,\infty}$ and $\| y \|_{p,\infty}\leq \| y \|_{p,q}$  for all $ y \in \ell^{p,q}$.
     Now, assume that $y\in B_{\ell^{p,q}}$. Then $y\in B_{\ell^{p,\infty}}$, so from the proof of the previous case, where  $q = \infty$, we again get $y_{n}^{*}\leq n^{-\frac{1}{p}}$ and $\sigma_{\ell^{p,q}} \leq  1+2^{-\frac{1}{p}} < 2\|I\|$.
    Now, applying Theorem~\ref{THM-not max noncompact} once more, we establish that $I$ is not maximally non-compact and 
    $$\alpha(I) \leq \varrho = 2^{-1}(1+2^{-\frac{1}{p}}) < 1.$$
   
\end{enumerate}
\end{proof}

\begin{proposition}
Let $p = \infty$ and $q \in (0, \infty)$. Then the embedding operator $I \colon \ell^{\infty, q} \rightarrow \ell^{\infty}$ is not maximally non-compact.
\end{proposition}

\begin{proof}
By Lemma~\ref{lorentz subset c_0}, $\ell^{\infty, q} \subset c_{0} \subset \ell^{\infty}$. Furthermore, $\|I\| = 1$ follows from the computations in Proposition~\ref{ex01:pq-I-c_0}.

We aim to bound the span  $\sigma_{\ell^{\infty, q}} =  \sup_{y\in B_{\ell^{\infty, q}}}(\sup y - \inf y)$ from above by $2\|I\| = 2$. Assume $y \in B_{\ell^{\infty, q}}$. Denote
\begin{align*}
    a & = \max(0, \sup y),\\
    b & = \max(0, -\inf y).
\end{align*}
If either $a = 0$ or $b = 0$, then 
$$\sup y - \inf y \le \max(\sup y, -\inf y) \le \|y\|_{\infty} \le \|y\|_{\infty, q} \le 1 < 2.$$

Assume now that both $a > 0$ and $b > 0$. Then $y_{1}^{*} \ge \max(a,b)$ and $y_{2}^{*} \ge \min(a,b)$. By the definition of the $\ell^{\infty,q}$ norm, we have
\begin{equation*}
    1 \ge \|y\|_{\infty, q}^{q} = \displaystyle\sum_{n=1}^{\infty} (y_{n}^{*})^{q} n^{-1} \ge (y_{1}^{*})^{q} + \frac{1}{2}(y_{2}^{*})^{q} \ge \max(a,b)^{q} + \frac{1}{2}\min(a,b)^{q}.
\end{equation*}
To find the upper bound of the span $\sigma_{\ell^{\infty, q}}$, we maximize the sum of these two components subject to this constraint. Let us define the set
$$K = \left\{ (u,v) \in \mathbb{R}^{2} \colon u \ge 0, v \ge 0, \max(u,v)^{q} + \frac{1}{2}\min(u,v)^{q} \le 1 \right\}.$$
Since $K$ is a closed and bounded subset of $\mathbb{R}^{2}$, it is compact. The function $f(u,v) = u + v$ is continuous, so it attains a maximum value on $K$, which we shall denote as $\sigma$.

We now show that $\sigma < 2$. If $\sigma \ge 2$, there would exist a point $(u,v) \in K$ such that $u + v \ge 2$. However, the condition $\max(u,v)^{q} \le 1$ implies $u \le 1$ and $v \le 1$. The only way to satisfy $u + v \ge 2$ under these restrictions is if $u = 1$ and $v = 1$. But for the point $(1,1)$, we have
$$ \max(1,1)^{q} + \frac{1}{2}\min(1,1)^{q} = 1 + \frac{1}{2} = \frac{3}{2} > 1, $$
which means $(1,1) \notin K$. This is a contradiction. Thus, $\sigma < 2$.

Consequently, for all $y \in B_{\ell^{\infty, q}}$, we know that $(a,b) \in K$, so we have
$$ \sup y - \inf y \le a + b \le \sigma < 2 = 2\|I\|. $$
This demonstrates that $\sigma_{\ell^{\infty, q}} < 2\|I\|$. By Theorem~\ref{THM-not max noncompact}, all prerequisites are satisfied, so $I$ is not maximally non-compact, and $\alpha(I) \le \frac{\sigma_{\ell^{\infty, q}}}{2} < 1$.
\end{proof}

\noindent \textbf{Acknowledgment.} This research was in part supported by the grant no.~23-04720S of the Czech Science Foundation.

The author would also like to thank both anonymous referees for their thorough and critical reading of the paper and for many very valuable remarks and suggestions for adjustments and further revision.

\bibliography{bibliography}

\

\end{document}